\documentclass{amsart}
\usepackage{amsmath,amsfonts,amssymb, bm}
\usepackage{graphicx,a4wide,comment,amscd}

\newtheorem{lemma}{Lemma}[section]

\newtheorem{prop}[lemma]{Proposition}
\newtheorem{thm}[lemma]{Theorem}
\newtheorem{cor}[lemma]{Corollary}
\theoremstyle{definition}
\newtheorem{defn}[lemma]{Definition}

\theoremstyle{remark}
\newtheorem{rmk}[lemma]{Remark}

\newcommand{\locdim}{\dim_{\mathrm{loc}}}
\newcommand{\llocdim}{{\underline{\dim}}_{\mathrm{loc}}}
\newcommand{\ulocdim}{{\overline{\dim}}_{\mathrm{loc}}}

\renewcommand{\rho}{\varrho}
\renewcommand{\epsilon}{\varepsilon}
\numberwithin{equation}{section}
\numberwithin{table}{section}

\begin{document}
\title[Local dimensions]{Local dimensions of overlapping self-similar measures}
\author{Kathryn E. Hare and Kevin G. Hare}
\address{Department of Pure Mathematics \\
University of Waterloo \\
Waterloo, Ontario \\
Canada N2L 3G1}
\thanks{Research of K. G. Hare was supported by NSERC Grant 2014-03154.}
\email{kehare@uwaterloo.ca}
\thanks{Research of K. E. Hare was supported by NSERC Grant 2016-03719.}
\email{kghare@uwaterloo.ca}
\date{\today }

\begin{abstract}
We show that any equicontractive, self-similar measure arising from the IFS
of contractions $(S_{j})$, with self-similar set $[0,1]$, admits an isolated
point in its set of local dimensions provided the images of $S_{j}(0,1)$
(suitably) overlap and the minimal probability is associated with one
(resp., both) of the endpoint contractions. Examples include $m$-fold
convolution products of Bernoulli convolutions or Cantor measures with
contraction factor exceeding $1/(m+1)$ in the biased case and $1/m$ in the
unbiased case. We also obtain upper and lower bounds on the set of local
dimensions for various Bernoulli convolutions.
\end{abstract}

\maketitle

\section{Introduction}

\label{sec:intro}

Consider the iterated function system (IFS) consisting of contractions of 
$S_{j}:[0,1]\rightarrow \lbrack 0,1],$ with common contraction factor $\rho ,$
and probabilities $p_{j},$ $j=0,...,m\geq 1$. By the equicontractive,
self-similar measure associated with this IFS we mean the unique Borel
probability measure $\mu $ satisfying 
\begin{equation}
\mu =\sum_{j=0}^{m}p_{j}\cdot \mu \circ S_{j}^{-1}.  \label{ssMeas}
\end{equation}
This measure is supported on the associated self-similar set and is well
known to be either purely absolutely continuous with respect to Lebesgue
measure or purely singular. When $S_{0}(x)=\rho x$, $S_{1}(x)=\rho x+1-\rho $
and $m=1$, the associated self-similar measures are known as Cantor measures
or Bernoulli convolutions, and are sometimes referred to as unbiased if 
$p_{0}=p_{1},$ or biased if $p_{0}\neq p_{1}$.

Our interest is in the local behaviour of these measures. The \emph{local
dimension} of a measure $\mu $ at a point $x$ in the support of $\mu $ is
defined as 
\begin{equation}
\locdim \mu (x)=\lim_{\epsilon \rightarrow 0}\frac{\log (\mu
([x-\epsilon ,x+\epsilon ]))}{\log \epsilon }.  \label{LocDimDefn}
\end{equation}
In the case that the IFS satisfies the open set condition, it is well known
that the set of local dimensions of the associated self-similar measure is a
closed interval and there are formulas for the endpoints of the interval
which depend on the contraction factors and the probabilities. See \cite{Fa}
for more details. For measures that do not satisfy the open set condition
the situation is much less well understood. In \cite{HuLau}, Hu and Lau
discovered that the $3$-fold convolution of the unbiased middle-third Cantor
measure admits an isolated point in its set of local dimensions. This was
later found to be true for certain other equicontractive self-similar
Cantor-like measures arising from IFS which have enough `overlap', such as
the $m$-fold convolution product of the unbiased Cantor measure with
contraction factor $1/m$, \cite{BHM, Sh}. These measures all had the
so-called `finite type' property, a separation condition permitting
overlaps, but stronger than the weak separation condition.

The Bernoulli convolutions with contraction factor the inverse of a Pisot
number\footnote{A Pisot number is a real algebraic integer, greater than $1$, such that all
of its Galois conjugates are strictly less than 1 in absolute value.} also
have the finite type property. These are particularly interesting being the
only known singular Bernoulli convolutions, see \cite{Erdos39, Shmerkin14,
Solomyak95}. There is a long history of studying the dimensionality
properties of these measures, c.f., \cite{PSS, So, Va} and the many
references cited therein for historical information. In \cite{F1, F2}, Feng
conducted a study of mainly unbiased Bernoulli convolutions with contraction
factor the inverse of a simple Pisot number and proved that for this class
of measures the set of local dimensions is an interval. In contrast, in 
\cite{HHN} it was shown that all biased Bernoulli convolutions with these
contraction factors admit an isolated point.

In this paper, we show that {\em any} equicontractive, self-similar
measure will admit an isolated point in its set of local dimensions provided
the images of $[0,1]$ under the contractions strictly overlap and $p_{0}$ is
the unique minimal probability (Theorem \ref{thm:general biased}). We also
prove there is an isolated point if $p_{0}=p_{m}$ are the unique minimal
probabilities and there is `sufficient' overlap (Theorem \ref{thm:general
unbiased}). In particular, we prove that if $\mu $ is any Bernoulli
convolution or Cantor measure with contraction factor $\rho >1/(m+1)$ in the
biased case and $\rho >(\sqrt{m^{2}+4}-m)/2$ in the unbiased case, then the 
$m$-fold convolution of $\mu $ with itself has an isolated point in its set
of local dimensions, improving upon the examples given in \cite{BHM, Sh}.

In all these cases, the isolated point is the local dimension at $0$. This
local dimension at $x = 0$ is easy to compute and is the maximum local dimension. A
challenging problem is to find sharp bounds for the set of local dimensions
at other $x$. Upper bounds have recently been found in \cite{Baker17} for
special classes of examples of these measures, including the biased
Bernoulli convolutions. Our arguments of Section \ref{sec:upper} also give
upper bounds. In Section \ref{sec:comp} we discuss other computational
techniques that allow us to prove even better upper bounds for the local
dimensions. A variation of these techniques are used in Section 
\ref{sec:lower} to find lower bounds for local dimension in the case where the
self-similar measure satisfies the asymptotically weak separation condition.
These techniques are applied to various Bernoulli convolutions.

\section{Terminology and Basic Properties}

\label{sec:term}

Throughout the paper we study the IFS $(S_{j},p_{j})$ consisting of the
contractions 
\begin{equation}
S_{j}(x)=\rho x+d_{j}\text{ for }0=d_{0}<d_{1}<\dots <d_{m}=1-\rho ,\ j=0,...,m,\text{ }m\geq 1  \label{IFS}
\end{equation}
and probabilities $p_{j}>0,$ $\sum_{j=0}^{m}p_{j}=1,$ and the associated
self-similar measure $\mu $ satisfying (\ref{ssMeas}). We further assume
that $d_i - d_{i-1} \leq \rho$ from which it follows that the associated
self-similar set (and hence support of $\mu$) is $[0,1]$.  We refer to $\rho $
as the contraction factor of the IFS or the self-similar measure. When $m=1$
the associated self-similar measure is a Cantor measure (when $\rho <1/2)$
or Bernoulli convolution (when $\rho >1/2)$. If, in addition, 
$p_{0}=p_{1}=1/2$ (the unbiased case) we often denote the Cantor measure or
Bernoulli convolution by $\mu _{\rho }$.

The notion of local dimension of a measure was stated in (\ref{LocDimDefn}).
Of course, the limit need not exist and when we replace the limit by $\lim
\sup $ or $\lim \inf $, then this gives the {\em upper and lower local
dimensions of $\mu $ at $x$} denoted by $\ulocdim \mu (x)$ and 
$\llocdim \mu(x)$ respectively:
\begin{eqnarray*}
\ulocdim \mu(x) &=&\limsup_{\epsilon \rightarrow 0}
\frac{\log (\mu ([x-\epsilon ,x+\epsilon ]))}{\log \epsilon }, \\
\llocdim \mu(x) &=&\liminf_{\epsilon \rightarrow 0}
\frac{\log (\mu ([x-\epsilon ,x+\epsilon ]))}{\log \epsilon }.
\end{eqnarray*}

Given $\sigma =(\sigma _{1},\sigma _{2},\dots ,\sigma _{n})\in \mathcal{A}^{n}$ where $\mathcal{A}=\{0,1,\dots ,m\}$, we denote by $S_{\sigma }$ the
concatenation $S_{\sigma _{1}}\circ \dots \circ S_{\sigma _{n}}$ and put 
$p_{\sigma }=p_{\sigma _{1}}p_{\sigma _{2}}\dots p_{\sigma _{n}}$. For 
$\sigma = (\sigma_1, \sigma_2, \sigma_3 \dots) \in \mathcal{A}^{\mathbb{N}}$ we define $S_{\sigma
}(0)=\lim_{n\rightarrow \infty }S_{\sigma_1 \sigma_2 \dots \sigma_n}(0)$.
We let 
\begin{equation*}
\mathcal{E}(x)=\{\sigma \in \mathcal{A}^{\mathbb{N}}:S_{\sigma }(0)=x\}.
\end{equation*}
If $\sigma =(\sigma _{i})\in \mathcal{E}(x)$ is a presentation of $x$, then 
$x=\sum \sigma _{i}\rho ^{i-1}(1-\rho )$, thus the set $\mathcal{E}(x)$ can
be thought of as the set of beta-expansions of $x$ with digit set
$\mathcal{A}$, as first introduced in \cite{Pa, Re}.

We also set 
\begin{align}
\mathcal{E}_{n}(x)& =\{\sigma \in \mathcal{A}^{n}:\mathrm{there\ exists}\
\tau \in \mathcal{A}^{\mathbb{N}}\ \mathrm{such\ that}\ S_{\sigma \tau
}(0)=x\}  \notag \\
& =\{\sigma \in \mathcal{A}^{n}:x\in S_{\sigma }([0,1])\}  \label{E}
\end{align}
and 
\begin{equation}
\mathcal{N}_{n}(x)=\sum_{\sigma \in \mathcal{E}_{n}(x)}p_{\sigma }\text{ .}
\label{N}
\end{equation}
Much is known about $\mathcal{E}_{n}(x)$ and $\mathcal{N}_{n}(x),$
especially when $\mu _{\rho }$ is a unbiased Bernoulli convolution, see 
\cite{FengSidorov09, Kempton18} for recent results. In particular, note that 
if $\sigma \in \mathcal{E}_{n}(x),$ then $S_{\sigma }([0,1])\subset \lbrack
x-\rho ^{n},x+\rho ^{n}]$. Thus $\mu ([x-\rho ^{n},x+\rho ^{n}])\geq 
\mathcal{N}_{n}(x)$ from which the following statement is immediate.

\begin{lemma}
\label{lem:3.1}Assume $\mu $ is an equicontractive self-similar measure with
contraction factor $\rho $. For all $x\in $supp$\mu $ we have 
\begin{equation*}
\ulocdim \mu(x) \leq \limsup_{n\rightarrow \infty }
\frac{\log (\mathcal{N}_{n}(x))}{n\log \rho }.
\end{equation*}
\end{lemma}

Together with an older result of Erd\"{o}s, we can quickly deduce that
unbiased Bernoulli convolutions with large enough contraction factors have
isolated points in their set of local dimensions. This is essentially proved
in \cite{EJK}, but we include a sketch here for completeness.

\begin{prop}
\label{Erdos}Let $\mu _{\rho }$ be the unbiased Bernoulli convolution with
contraction factor $\rho >(\sqrt{5}-1)/2$. Then
\begin{equation*}
\sup_{x\in (0,1)}
\ulocdim \mu_\rho (x) <\locdim \mu _{\rho }(0)=\locdim \mu _{\rho }(1).
\end{equation*}
\end{prop}

\begin{proof}
In the proof of Theorem 3 of \cite{EJK} it is shown that if $k$ is chosen
such that $1<\rho ^{2}+\rho ^{3}+\dots +\rho ^{k}$, then for any $x\in
(0,1), $ we have $\#\mathcal{E}_{n}(x)\geq c(x)2^{n/k}$ for some $c(x)>0$
and independent of $n$. Thus $\mathcal{N}_{n}(x)\geq c(x)2^{-n(1-1/k)}$.
Appealing to the previous lemma it follows that for all $x\neq 0,1$,
\begin{equation*}
\ulocdim \mu_{\rho}(x) \leq \limsup_{n\rightarrow \infty }\frac{\log (c(x)2^{-n(1-1/k)})}{n\log \rho }\leq \left(1-\frac{1}{k}\right) \frac{\log 2}{|\log \rho |}.
\end{equation*}

The conclusion of the proposition holds since $\locdim \mu _{\rho
}(0)=\locdim \mu _{\rho }(1)=\log 2/\left\vert \log \rho \right\vert$.
\end{proof}

\section{Isolated points in the set of local dimensions}

\label{sec:upper}

We will say the IFS of (\ref{IFS}) has {\em strict overlap} if
$S_{j}(1)>S_{j+1}(0)$ for each $j=0,...,m-1$. Equivalently, $S_{j}(0,1)\cap
S_{j+1}(0,1)\neq \emptyset $ for $j=0,\dots ,m-1$. An example is the IFS
generating the Bernoulli convolution with contraction factor $\rho >1/2$.

\begin{thm}
\label{thm:general biased} Suppose $\mu $ is an equicontractive,
self-similar measure associated with the IFS $(S_{j},p_{j})$ of (\ref{IFS})
that has the strict overlap property. If $p_{0}<p_{j}$ for all $j\neq 0$,
then \begin{equation*}
\sup_{x\neq 0}\ulocdim \mu (x)<\locdim \mu (0)
\end{equation*}
and thus $\locdim \mu (0)$ is an isolated point in the set 
$\{\locdim\mu (x):x\in $supp$\mu \}.$
\end{thm}

\begin{proof}
The strict overlap property is equivalent to the inequalities $d_{j-1}+\rho
>d_{j}$ for all $j=1,\dots ,m$, thus we can choose $0<\xi <\rho $ so 
$d_{j-1}+\rho >d_{j}+\xi $ for all $j=1,\dots ,m$. Choose an integer $J>0$
such that $\rho ^{J}<\xi $.

We claim any $x\neq 0$ has a presentation $(a_{k})$ where the density of
indices $k$ with $a_{k}\neq 0$ exceeds $1/J$. Assume $p_{i}=\min_{j\neq
0}p_{j}$. It follows from the claim that if $x\neq 0$ and $n$ is large, then 
$\mathcal{N}_{n}(x)\geq \left( p_{i}^{1/J}p_{0}^{(J-1)/J}\right) ^{n},$ and
hence by Lemma \ref{lem:3.1} 
\begin{equation*}
\ulocdim \mu(x) \leq \frac{\log p_{i}+(J-1)\log p_{0}}
{J\log \rho }<\frac{\log p_{0}}{\log \rho }=\locdim \mu (0),
\end{equation*}
proving the result.

To prove the claim, we will give an iterative algorithm for producing such a
presentation. This algorithm is essentially the lazy expansion of $x$ with
respect to the alphabet $\mathcal{A}=\{0,1,...,m\}$. To begin, if $x\in
\lbrack 0,d_{1}+\xi ]\subseteq S_{0}[0,1]$ choose $a_{1}=0$; if $x\in
(d_{j}+\xi ,$ $d_{j+1}+\xi ]$ $\subseteq S_{j}[0,1]$ for $j=1,\dots ,m-1$
take $a_{1}=j$; and if $x\in (d_{m-1}+\xi ,1]\subseteq S_{m}[0,1]$ take 
$a_{1}=m$.

Assuming $a_{1},\dots ,a_{N}$ have been chosen, set $\sigma =(a_{1},\dots
,a_{N})$. Then $x\in S_{\sigma }[0,1]$. We put $a_{N+1}=0$ if $x\in
S_{\sigma }[0,d_{1}+\xi ]$, $a_{N+1}=j$ if $x\in S_{\sigma }(d_{j}+\xi , d_{j+1}+\xi ]$ and $a_{N+1}=m$ if $x\in S_{\sigma }(d_{m-1}+\xi ,1]$.

Suppose $x\neq 0$ has presentation $(a_{k})$ under this algorithm. As $x\neq
0$, there is some index $n$ such that $a_{n}\neq 0,$ say $a_{n}=j$ for 
$j\neq 0$. We will see that it is not possible for all of $a_{n+1},\dots
,a_{n+J}=0$. Without loss of generality $j\neq m$. (The arguments are
similar when $j=m$ and will be left for the reader.) Put $\sigma
=(a_{1},\dots ,a_{n-1})$. Then $x\in S_{\sigma }(d_{j}+\xi ,$ $d_{j+1}+\xi ]$,
hence $x-S_{\sigma }(d_{j})\geq \xi \rho ^{|\sigma |}$. If all $a_{n+j}=0$
for $j=1,\dots ,J$, then 
\begin{equation*}
x\in S_{\sigma j\underbrace{0\cdot \cdot \cdot 0}_{J-1}}[0,d_{1}+\xi
]\subseteq S_{\sigma j\underbrace{0\cdot \cdot \cdot 0}_{J-1}}[0,\rho ].
\end{equation*}
Thus 
\begin{eqnarray*}
x &\leq &\sup S_{\sigma j\underbrace{0\cdot \cdot \cdot 0}_{J-1}}[0,\rho
]=S_{\sigma }(d_{j})+\rho ^{J+\left\vert \sigma \right\vert } \\
&<&S_{\sigma }(d_{j})+\xi \rho ^{\left\vert \sigma \right\vert },
\end{eqnarray*}
which is a contradiction.
\end{proof}

Note that the proof actually establishes that 
\begin{equation*}
\sup_{x\neq 0}\ulocdim \mu (x)\leq \frac{\log
(\min_{j\neq 0}p_{j})+(\frac{\log \xi }{\log \rho }-1)\log p_{0}}{\log \xi }
\end{equation*}
where $\xi =\min_{j=1,....,m}(d_{j-1}+\rho -d_{j})$. The corollary below is
a special case of this.

\begin{cor}
If $\mu _{\rho }$ is a biased Bernoulli convolution with $\rho >
(\sqrt{5}-1)/2$ and $p_{0}<p_{1}$, then
\begin{equation*}
\sup_{x\neq 0}\ulocdim \mu _{\rho }(x)\leq 
\frac{2/3\log p_{0}+1/3\log (1-p_{0})}{\log \rho }.
\end{equation*}
\end{cor}

\begin{proof}
In the notation of the proof of Theorem \ref{thm:general biased} we can choose any $\xi <2\rho-1$. As $\rho ^{3}=2\rho -1$ when $\rho =(\sqrt{5}-1)/2$, it follows that
we can choose $\xi >\rho ^{3}$.
\end{proof}

\begin{rmk}
\label{rmk:biasedconv}The $m$-fold convolution of a measure can be defined
inductively as $\mu ^{m}=\mu ^{m-1}\ast \mu $. For the purposes of this
paper, we typically rescale the convolution so that $\mu ^{m}$ still has
support in $[0,1]$. This will not affect dimensionality results. When $\mu $
arises from an equicontractive IFS, then $\mu ^{m}$ is again a self-similar
measure. For example, if $\mu $ is the biased Bernoulli convolution or
Cantor measure with contraction factor $\rho $ and probabilities $p,1-p$,
then $\mu ^{m}$ is the self-similar measure associated with the IFS 
$S_{j}(x)=\rho x+j(1-\rho )/m$ and probabilities 
$p_{j}=\binom{m}{j} p^{j}(1-p)^{m-j}$ for $j=0,...,m$. It is easy to check that this IFS has the
strict overlapping property if $\rho >1/(m+1)$ and thus will have an
isolated point if $p\neq 1-p$.
\end{rmk}

If a stricter overlapping property is satisfied, more can be proven.

\begin{thm}
\label{thm:general unbiased} Suppose $\mu $ is an equicontractive,
self-similar measure associated with the IFS $(S_{j},p_{j})$ of (\ref{IFS})
that has the strict overlap property. Suppose $m\geq 2$ and 
$p_{0}=p_{m}<p_{j}$ for all $j\neq 0,m$. In addition, assume that 
$S_{m-1}(1)>S_{m}S_{1}(0)$ and $S_{1}(0)<S_{0}S_{m-1}(1)$. Then 
\begin{equation*}
\sup_{x\neq 0,1}\ulocdim \mu (x)<\locdim \mu (0)
\end{equation*}
and thus $\locdim \mu (0)=$ $\locdim\mu (1)$ is an
isolated point in the set of local dimensions of $\mu $.
\end{thm}

\begin{proof}
The additional overlapping condition, $S_{m-1}(1)>S_{m}S_{1}(0)$ and 
$S_{1}(0)<S_{0}S_{m-1}(1),$ is equivalent to the two inequalities $\rho
(d_{m-1}+\rho )>d_{1}$ and $d_{m}+\rho d_{1}<d_{m-1}+\rho $. Thus we can
choose $\xi >0$ such that for each $j$,

\begin{enumerate}
\item $d_{j}+\rho -\xi >d_{j+1}+\xi $

\item $\rho (d_{m-1}+\rho -\xi )>d_{1}+\xi $

\item $d_{m}+\rho (d_{1}+\xi )<d_{m-1}+\rho -\xi .$
\end{enumerate}

Let
\begin{equation*}
L=[0,d_{1}+\xi ),R=(d_{m-1}+\rho -\xi ,1],
\end{equation*}
\begin{equation*}
M_{j}=[d_{j}+\xi ,d_{j}+\rho -\xi ]\text{ for }j=1,\dots ,m-1
\end{equation*}
and $M=\bigcup\limits_{j=1}^{m-1}M_{j}$. Then $[0,1]=L\cup M\cup R$.
One should observe that properties (1)-(3) ensure that $\bigcup
\limits_{\sigma }S_{\sigma }(M)=(0,1)$ where the union is taken over all words 
$\sigma$ on the alphabet $\{0,1,\dots ,m\}$.

As in the previous proof, we claim that if $\rho ^{J}<\xi $, then any $x\neq
0,1$ has a presentation $(a_{i})$ where the density of indices $i$ with 
$a_{i}\neq 0,m$ is at least $1/J$.

We use the following algorithm to produce the presentation: Take $a_{1}=j$
if $x\in M_{j}$ (if there is a non-unique choice, choose either index). If 
$x\notin \bigcup\limits_{j}M_{j}$, then either $x\in L$ or $R$ and we take 
$a_{1}=0$ or $m$ respectively. Now assume $a_{1},\dots ,a_{n-1}$ have been
determined and put $\sigma =(a_{1},\dots ,a_{n-1})$, so $x\in S_{\sigma
}[0,1]$. If $x\in S_{\sigma }(M_{j})$ take $a_{n}=j,$ while if $x\in
S_{\sigma }(L)$ or $S_{\sigma }(R)$ take $a_{n}=0,m$ respectively.

If $x\neq 0,m$ then there must be an index $n$ where $a_{n}=j\in \{1,\dots
,m-1\}$. We claim that $a_{n+k}\neq 0,m$ for some $k<J$. That is, we cannot
have all of $a_{n+1}, a_{n+2}, \dots, a_{n+J} \in \{0,m\}$.

To prove this claim, put $\sigma =(a_{1},\dots ,a_{n-1})$. As $a_{n}=j$, we
have
\begin{equation*}
x\in S_{\sigma }(M_{j})\subseteq S_{\sigma j}[0,1]=S_{\sigma j}(L)\cup
S_{\sigma j}(M)\cup S_{\sigma j}(R).
\end{equation*}
If $x\in S_{\sigma j}(M),$ then $x\in S_{\sigma j}(M_{k})$ for some $k\neq
0,m$ and we are done since that ensures $a_{n+1}\neq 0,m$. So we can assume 
$x\in S_{\sigma j}(L)$ or $S_{\sigma j}(R)$.

We will assume that $x \in S_{\sigma j}(L)$; the latter case is similar.
Hence $a_{n+1}=0$. Upon rescaling we can assume 
\begin{equation*}
x\in M_{j}=[d_{j}+\xi ,d_{j}+\rho -\xi ]\subseteq \lbrack d_{j},d_{j}+\rho
]=S_{j}[0,1]
\end{equation*}
and $x\in S_{j}(L)$ where
\begin{eqnarray*}
S_{j}(L) &=&[d_{j},d_{j}+\rho (d_{1}+\xi ))\subseteq \lbrack
d_{j},d_{j}+\rho ^{2}] \\
&=&S_{j}[0,\rho ]=S_{j0}[0,1]=S_{j0}(L\cup M\cup R).
\end{eqnarray*}
But 
\begin{equation*}
\inf S_{j0}(R)=S_{j0}(d_{m-1}+\rho -\xi )=d_{j}+\rho ^{2}(d_{m-1}+\rho -\xi
).
\end{equation*}
Property (2) thus implies that $\sup S_{j}(L)<\inf S_{j0}(R)$ and hence we
must actually have $x\in S_{j0}(L\cup M)$. Thus $a_{n+2}\neq m$. If 
$a_{n+2}\neq 0$ we are done and otherwise $x\in S_{j0}(L)$. We repeat the
argument. Since $S_{j\underbrace{0\cdot \cdot \cdot 0}_{J-1}}(L)$ is an
interval of length at most $\rho ^{J}$ with left endpoint $d_{j}$ and $x\geq
d_{j}+\xi $, we see that if $\rho ^{J}<\xi $ we cannot have $x\in 
S_{j \underbrace{0\cdot \cdot \cdot 0}_{J-1}}(L)$. This completes the proof.
\end{proof}

\begin{cor}
Suppose $\mu $ is the self-similar measure associated with the IFS 
\begin{equation*}
S_{j}(x)=\rho x+\frac{j}{m}(1-\rho ),
\end{equation*}
with $j=0,...,m,$ $m\geq 2,$ and probabilities $p_{j}$ that satisfy 
$ p_{0}=p_{m}<\min_{j\neq 0,k}p_{j}$. If 
\begin{equation*}
\rho >\frac{\sqrt{m^{2}+4}-m}{2},
\end{equation*}
then the set of local dimensions of $\mu $ has an isolated point.
\end{cor}

\begin{proof}
A routine calculation shows that the overlap requirements of Theorem
\ref{thm:general unbiased}, $d_{j}<d_{j-1}+\rho ,$ $\rho (d_{m-1}+\rho )>d_{1}$
and $d_{m}+\rho d_{1}<d_{m-1}+\rho ,$ are satisfied for such $\rho $.
\end{proof}

We remark that $(\sqrt{m^{2}+4}-m)/2<1/m,$ so this improves upon the fact
that the local dimension of the $m$-fold convolution of the uniform Cantor
measure on the Cantor set with ratio $1/m$ has an isolated point at $0,$ as
shown in \cite{BHM, Sh}. Note that when $\rho =1/(m+1),$ the IFS satisfies
the open set condition and hence there is no isolated point.

\section{Computational techniques to find upper and lower bounds}

\subsection{Upper bounds}

\label{sec:comp}

Proposition \ref{Erdos} and Theorem \ref{thm:general biased} give upper
bounds on the local dimension of the Bernoulli convolution $\mu _{\rho }$ at 
$x$ for any $x\in (0,1)$ in the unbiased ($\rho >(\sqrt{5}-1)/2)$ and biased
($\rho >1/2)$ cases respectively. Upper bounds are also given in 
\cite{Baker17}. In all of these cases, these bounds can be used to show that the
local dimension at $x=0$ is an isolated point within the set of all possible
local dimensions. However, while sufficient to demonstrate a gap in the set
of local dimensions, these bounds are not tight. This section will discuss
computational techniques that can be used to improve the upper bounds for
the set of local dimensions for $x\in (0,1)$. We do this for both the
unbiased and biased Bernoulli convolutions.

In the case of the unbiased Bernoulli convolution, we see that $
\locdim \mu _{\rho }(0)=\log 2/|\log \rho |$. This is given by the blue
curve in Figure \ref{fig:gap}. Further, as shown in Proposition \ref{Erdos}
(and \cite{EJK}), taking $k$ such that $1<\rho ^{2}+\dots +\rho ^{k}$, gives 
\begin{equation*}
\ulocdim \mu _{\rho }(x)\leq \left( 1-\frac{1}{k} \right) \frac{\log 2}{|\log \rho |}\text{ for }x\in (0,1).
\end{equation*}
This formula is shown by the red curve in Figure \ref{fig:gap}. The black
curve in Figure \ref{fig:gap} shows the tighter upper bound given by an
application of Theorem \ref{thm:upper} below.

Now consider the case of the biased Bernoulli convolutions. Assume 
$p_{0}<p_{1}$. It is easy to see that $\locdim \mu _{\rho
}(0)=\log p_{0}/\log \rho $. This is given by the blue curve in Figure 
\ref{fig:biased gap}. Both Theorem \ref{thm:general biased} and Baker in 
\cite{Baker17}, show there exists some $k$ such that 
\begin{equation*}
\ulocdim \mu_{\rho }(x)\leq \frac{\log
p_{0}+(k-1)p_{1}}{k\log \rho }.
\end{equation*}
The choice of $k$ varies in the two approaches and depends on $\rho $. When 
$\rho <\frac{\sqrt{5}-1}{2}$, the $k$ found by Theorem \ref{thm:general
biased} results in a tighter upper bound than that found in \cite{Baker17},
while for $\rho >\frac{\sqrt{5}-1}{2}$, the converse is true. The green
curve in Figure \ref{fig:biased gap} shows the upper bound given by Theorem 
\ref{thm:general biased}, while the red curve shows the bound found in 
\cite{Baker17}. The black curve is, again, the upper bound found using Theorem 
\ref{thm:upper}.

\begin{figure}[tbp]
\includegraphics[width=300pt,height=400pt,angle=270]{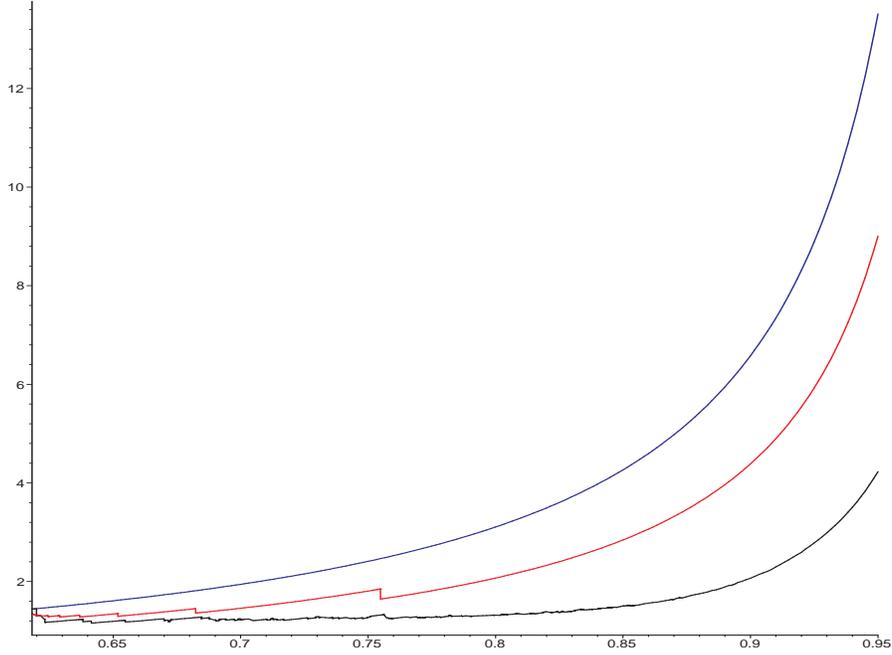}
\caption{Upper Bounds for local dimensions of unbiased Bernoulli
convolutions }
\label{fig:gap}
\end{figure}

\begin{figure}[tbp]
\includegraphics[width=300pt,height=400pt,angle=270]{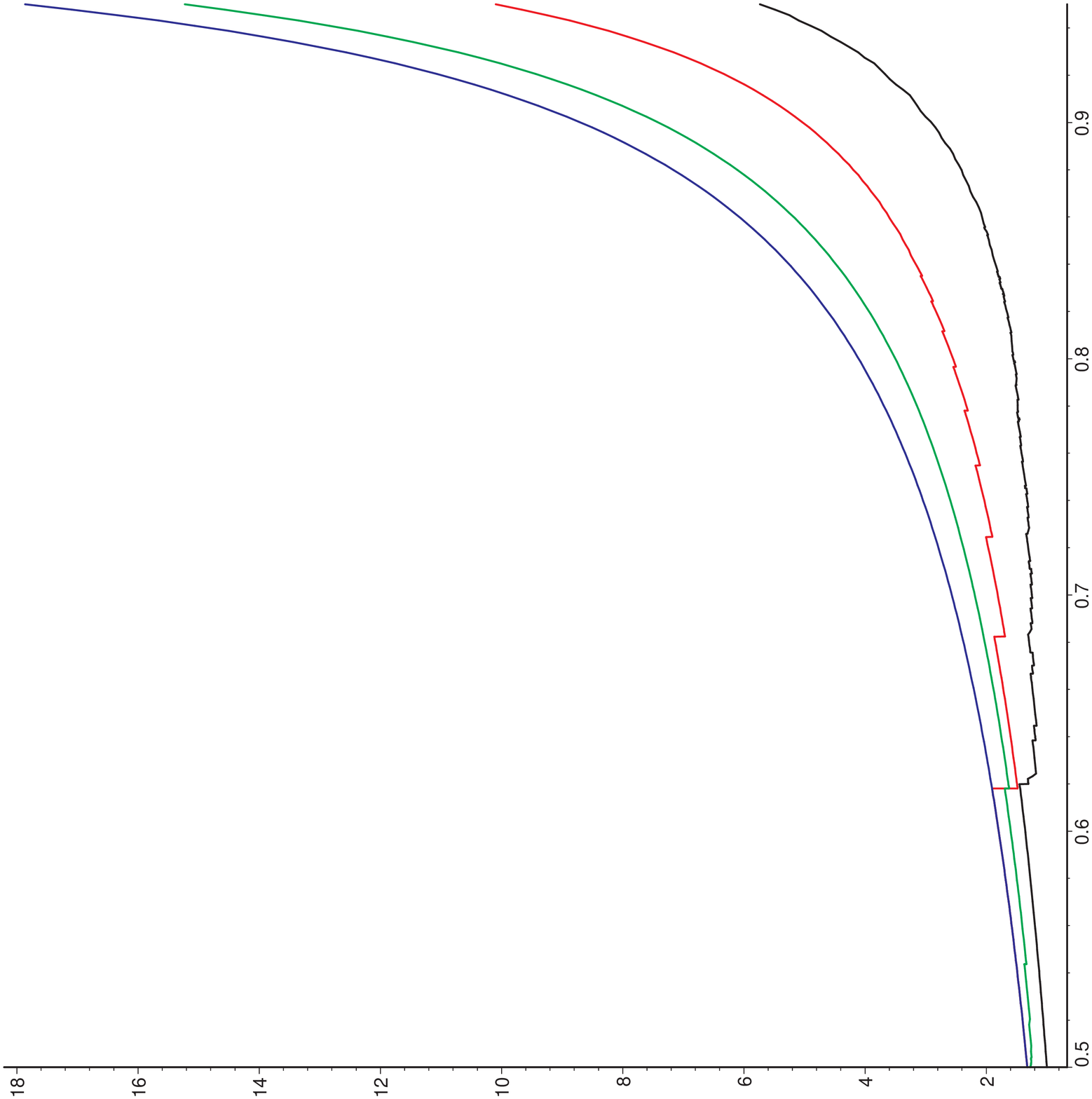}
\caption{Upper Bounds for local dimensions of biased $(p_0 = 0.4, p_1 = 0.6)$
Bernoulli convolutions}
\label{fig:biased gap}
\end{figure}

Let $\mu$ be a self-similar measure with support $[0,1]$.
Let $I$ be a subset of $[0,1]$. We generalize equation (\ref{E}) to give 
\begin{eqnarray*}
\mathcal{E}_{n}(x,I) &=&\{\sigma \in \mathcal{A}^{n}:\ \text{there exists a}
\ \tau \in \mathcal{A}^{\mathbb{N}}\ \text{such that}\ S_{\sigma \tau
}(0)=x\ \text{and}\ S_{\tau }(0)\in I\} \\
&=&\{\sigma \in \mathcal{A}^{n}:x\in S_{\sigma }(I)\}.
\end{eqnarray*}
We similarly generalize equation (\ref{N}) to give 
\begin{equation*}
\mathcal{N}_{n}(x,I)=\sum_{\sigma \in \mathcal{E}_{n}(x,I)}p_{\sigma }.
\end{equation*}
We easily see that $\mathcal{N}_{n}(x,I)\leq \mathcal{N}_{n}(x)$ for all 
$x,I $ and $n$.

\begin{thm}
Let $\mu$ be a self-similar measure with support $[0,1]$.
\label{thm:upper} Let $I\subset \lbrack 0,1]$ be an open interval such that
for all $x\in (0,1)$ there exists a word $\sigma $ with $x\in S_{\sigma }(I)$. 
Let $k=\min_{x\in I}\mathcal{N}_{n}(x,I)$. Then 
\begin{equation*}
\ulocdim \mu(x)\leq \frac{\log k}{n\log
\rho }.
\end{equation*}
\end{thm}

For a given measure $\mu$, if we can find an interval $I$ that satisfies
the requirements of the theorem and can compute $k$, then we will have a
computational method to find an upper bound for $\ulocdim \mu (x)$ for $x\in (0,1)$.

\begin{proof}
First, assume that $x\in I$. We will further assume that $k > 0$, otherwise
the bound is trivial. Hence there exists at least one $\sigma$ such that $x
\in S_{\sigma}(I)$. This gives us that 
\begin{eqnarray*}
\mathcal{E}_{2n}(x,I) &=&\{\sigma\in \mathcal{A}^{2n}:x\in S_{\sigma}(I)\} \\
&=&\{\sigma _{1},\sigma _{2}\in \mathcal{A}^{n}:x\in S_{\sigma _{1}\sigma
_{2}}(I)\} \\
&\supseteq &\{\sigma _{1},\sigma _{2}\in \mathcal{A}^{n}:x\in S_{\sigma
_{1}\sigma _{2}}(I),x\in S_{\sigma _{1}}(I)\} \\
&=&\{\sigma _{1},\sigma _{2}\in \mathcal{A}^{n}:x\in S_{\sigma
_{1}}(I),S_{\sigma _{1}}^{-1}(x)\in S_{\sigma _{2}}(I)\} \\
&=&\{\sigma _{1},\sigma _{2}:\sigma _{1}\in \mathcal{E}_{n}(x,I),\sigma
_{2}\in \mathcal{E}_{n}(S_{\sigma _{1}}^{-1}(x),I)\}.
\end{eqnarray*}
Note that if $x\in S_{\sigma _{1}}(I),$ then $S_{\sigma _{1}}^{-1}(x)$ is
well defined and in $I$.

From this it follows that 
\begin{equation*}
\mathcal{N}_{2n}(x)\geq \mathcal{N}_{2n}(x,I)\geq \sum_{\sigma _{1}\in 
\mathcal{E}_{n}(x,I)}\sum_{\sigma _{2}\in \mathcal{E}_{n}(S_{\sigma
_{1}}^{-1}(x),I)}p_{\sigma _{1}}p_{\sigma _{2}}\geq k^{2}.
\end{equation*}
Similarly $\mathcal{N}_{mn}(x)\geq k^{m}$. Furthermore for $mn\geq N\geq
(m-1)n$ we have $\mathcal{N}_{N}(x)\geq k^{m}$. Taking limits gives 
\begin{equation*}
\ulocdim \mu(x) \leq 
\limsup_{N\rightarrow \infty }\frac{\log \mathcal{N}_{N}(x)}{\log \rho ^{N}}
\leq \lim_{m\rightarrow \infty }\frac{\log k^{m}}{\log \rho ^{n(m-1)}}\leq 
\frac{\log k}{n\log \rho }.
\end{equation*}

The case for $x\in (0,1)\diagdown I$ is similar. We know there is some 
$\sigma ,$ say of length $t,$ such that $x\in S_{\sigma }(I)$. As $S_{\sigma
}^{-1}(x)\in I$, we have that $\mathcal{N}_{nm}(S_{\sigma }^{-1}(x),I)\geq
k^{m}$ as above. Thus $\mathcal{N}_{mn+t}(x,I)\geq p_{\sigma }k^{m},$ and
the result follows as before, taking limits.
\end{proof}

Next, we will demonstrate how the theorem can be implemented to produce
upper bounds on local dimensions by means of an example. Consider the
unbiased Bernoulli convolution with contraction factor $\rho =0.8$ and let 
$I=(0.3,0.7)$. One can check that $\bigcup_{|\sigma |=1}S_{\sigma
}(I)=(0.24,0.76)$ and in general that $\bigcup_{|\sigma |=n}S_{\sigma
}(I)=(0.3(0.8)^{n},1-0.3(0.8)^{n})$. It follows that the hypothesis of the
theorem is satisfied. There are 16 images of $S_{\sigma }(I)$ for $|\sigma
|=4$; these are given in Table \ref{tab:example}. 
\begin{table}[tbp]
\begin{tabular}{l|l}
$\sigma$ & $S_\sigma([0.3, 0.7]$ \\ \hline
{}0000 & [.12288, .28672] \\ 
{}0001 & [.22528, .38912] \\ 
{}0010 & [.25088, .41472] \\ 
{}0100 & [.28288, .44672] \\ 
{}1000 & [.32288, .48672] \\ 
{}0011 & [.35328, .51712] \\ 
{}0101 & [.38528, .54912] \\ 
{}0110 & [.41088, .57472] \\ 
{}1001 & [.42528, .58912] \\ 
{}1010 & [.45088, .61472] \\ 
{}1100 & [.48288, .64672] \\ 
{}0111 & [.51328, .67712] \\ 
{}1011 & [.55328, .71712] \\ 
{}1101 & [.58528, .74912] \\ 
{}1110 & [.61088, .77472] \\ 
{}1111 & [.71328, .87712] 
\end{tabular}
\caption{Images of $S_\protect\sigma([0.3, 0.7])$ for $|\protect\sigma|=4$}
\label{tab:example}
\end{table}
One can readily check that for each $x$ $\in (0.3,0.7)$ there are at least 
$3$ words $\sigma $ such that $x\in S_{\sigma }(I)$. Thus $\mathcal{N}_{4}(x,(0.3,0.7))\geq 3/16$. Hence $k\geq 3/16$ and $\ulocdim \mu_\rho (x)\leq \frac{\log (3/16)}{4\log (0.8)} \sim 1.876$.

These calculations while exact, are also locally constant. It can be shown
that there is a neighbourhood around $\rho =0.8$ and around the endpoint $0.3
$ and $0.7$ such that $k\geq 3/16$ within this neighbourhood. This comment
is true in general, not just for the special case of $\rho = 0.8$, $I=[0.3,0.7]$ and $n=4$.

This process can be generalized to other Bernoulli convolutions and
automated. Consider, first, the interval $I=(a,1-a)$ where $\rho a+1-\rho
<1/2$ and $\rho (1-a)>1/2$. (For example, take any $a$ satisfying 
$0<a<1-1/(2\rho ).$) Then $S_{0}(I)\cup S_{1}(I)=(\rho a,1-\rho a)$ and,
more generally, $\bigcup_{|\sigma |=n}S_{\sigma }(I)=(\rho ^{n}a,1-\rho
^{n}a)$. It is clear that the hypothesis of the theorem is satisfied for
such $I$.

Consider, next, an interval $I=(b,1-b)$ where $b$ may be larger than 
$1-1/(2\rho )$. If there exists a choice of $a$ with $0<a<1-1/(2\rho )$ and
integer $n$ such that $(a,1-a)\subseteq $ $\bigcup_{|\sigma |=n}S_{\sigma
}(I),$ then again $I$ will satisfy the hypothesis of the theorem.

For the purposes of the graphs, we considered the intervals $(\frac{1}{2}
(1-1/(2\rho)), 1-\frac{1}{2}(1-1/(2\rho))) =
(\frac{1}{2} - \frac{1}{4\rho}, \frac{1}{2} + \frac{1}{4\rho})$, 
$(0.1,0.9)$, $(0.2,0.8)$ and $(0.3,0.7)$. The first always
satisfies the conditions of the theorem and we compute the associated $k$.
The other three may or may not depending on whether we can find a choice of 
$n \leq 10$ as above where we view these intervals as the choice $(b,1-b)$
and understand the first interval as $(a,1-a)$. If we can quickly find a
suitable $n$, we compute the associated $k$. Otherwise we ignore the
interval. We take the minimum $k$ resulting from these choices of intervals.

A similar method can be used for any self-similar measure with non-trivial
overlaps, with the details being left to the reader.

\subsection{Lower bounds}

\label{sec:lower}

In Section \ref{sec:comp}, we showed how one could use computational
techniques to find upper bounds for the local dimension for $\mu_{\rho }(x)$
for any $x\in (0,1)$. Similar techniques can be used to find lower bounds
for the range of possible local dimensions assuming the IFS satisfies a
suitable separation condition.

The IFS (or any associated self-similar measure) with contraction factor 
$\rho $ is said to satisfy the {\em weak separation condition} (wsc) if
there is a constant $c>0$ such that whenever $\sigma ,\tau \in \mathcal{A}^{n}$, then either 
\begin{equation}
S_{\sigma }(0)=S_{\tau }(0)\text{ or }\left\vert S_{\sigma }(0)-S_{\tau
}(0)\right\vert \geq c\rho ^{n}.
\end{equation}

The following definition is equivalent to that of \cite{FengSalem}.

\begin{defn}
A equicontractive IFS with ratio of contraction $\rho $ satisfies the 
{\em asymptotically weak separation condition} (asymptotically wsc) if
there exists a sequence $f(n)$ such that $\log f(n)/n\rightarrow 0$ as 
$n\rightarrow \infty $, and such that for each $n\in \mathbb{N}$ and each 
$x\in \lbrack 0,1]$ we have 
\begin{equation}
\#\{S_{\sigma }[0,1]:\sigma \in \mathcal{A}^{n},S_{\sigma }[0,1]\cap (x-\rho
^{n},x+\rho ^{n})\neq \emptyset \}\leq f(n)  \label{awsc}
\end{equation}
\end{defn}

The weak separation condition implies the asymptotically wsc, with the
latter being strictly weaker. Indeed, as observed in \cite{FengSalem}, a
Bernoulli convolution with contraction factor the reciprocal of a Salem
number\footnote{A Salem number is a real algebraic integer, greater than $1$, such that all of its Galois conjugates $\leq 1$ in absolute value and at least one
conjugate is $1$ in absolute value.} in $(1,2)$ satisfies the asymptotically
wsc, but not the wsc. It is widely believed that if there are any
contraction factors which give rise to purely singular Bernoulli
convolutions other than reciprocals of Pisot numbers, then the prime
candidates would be reciprocals of Salem numbers.

It is worth noting that if $\rho \in (1/2, 1)$ is transcendental then 
$S_\sigma(0) \neq S_\tau(0)$ for all $\sigma \neq \tau$ and hence all images
on the left hand side of \eqref{awsc} are unique. Hence if $\rho$ is
transcendental then the IFS canot satisfy the asympototically  weak
separation condition.

We can obtain lower bounds in the spirit of Lemma \ref{lem:3.1} under the
assumption of the asymptotically wsc.

\begin{lemma}
Assume $\mu $ is an equicontractive, self-similar measure with contraction
factor $\rho $ that satisfies the asymptotically weak separation condition.
Then for all $x\in $supp$\mu $ we have 
\begin{equation*}
\llocdim \mu(x) \geq \liminf_{n\rightarrow \infty }
\frac{\log (\sup_{y}\mathcal{N}_{n}(y))}{n\log \rho }.
\end{equation*}
\end{lemma}

\begin{proof}
Suppose $f(n)$ is a sequence with $(\log f(n))/n\rightarrow 0$ and
satisfying (\ref{awsc}).

It is convenient to put 
\begin{equation*}
H_{n}=\{\sigma \in \mathcal{A}^{n} : S_{\sigma }[0,1]\cap \lbrack x-\rho
^{n},x+\rho ^{n}]\neq \emptyset \}.
\end{equation*}
With this notation, we have $\mu \lbrack x-\rho ^{n},x+\rho ^{n}]\leq
\sum_{\sigma \in H_{n}}p_{\sigma }$ for all $n$. Let 
\begin{equation*}
\mathcal{H}{_{n}}=\{S_{\sigma }(0):\sigma \in H_{n}\}.
\end{equation*}
Note that $\mathcal{H}{_{n}}$ may contain fewer elements than $H_{n}$ as we
may have $S_{\sigma }(0)=S_{\tau }(0)$ for $\sigma ,\tau \in H_{n}$. Indeed,
the asymptotically wsc guarantees $\#\mathcal{H}_{n}\leq f(n)$. Since 
$\sum_{\sigma \in \mathcal{A}^{n}:S_{\sigma }(0)=y}p_{\sigma }\leq \mathcal{N}_{n}(y),$ we see that 
\begin{equation*}
\mu \lbrack x-\rho ^{n},x+\rho ^{n}]\leq \sum_{y\in \mathcal{H}{_{n}}}
\mathcal{N}_{n}(y)\leq \#\mathcal{H}{_{n}}\sup_{y}\mathcal{N}_{n}(y)\leq
f(n) \sup_{y}\mathcal{N}_{n}(y).
\end{equation*}
Thus 
\begin{align*}
\llocdim \mu (x)& =\liminf_{n\rightarrow \infty }
\frac{\log (\mu \lbrack x-\rho ^{n},x+\rho ^{n}])}{\log (2\rho ^{n})} \\
& \geq \liminf_{n\rightarrow \infty }\frac{\log (\sup_{y}\mathcal{N}
_{n}(y))+\log (f(n))}{n\log \rho + \log 2 } \\
& =\liminf_{n\rightarrow \infty }\frac{\log \sup_{y}(\mathcal{N}_{n}(y))}
{n\log \rho }.
\end{align*}
\end{proof}

Note that $\sup_{y\in \lbrack 0,1]}\mathcal{N}_{n+m}(y)\geq \sup_{y\in
\lbrack 0,1]}\mathcal{N}_{n}(y)\times \sup_{y\in \lbrack 0,1]}\mathcal{N}
_{m}(y)$. Thus good bounds on $\sup_{y\in \lbrack 0,1]}\mathcal{N}_{n}(y)$
will result in good lower bounds for $\llocdim \mu
(x)$.

We will consider the same example as in the previous subsection. Let $\rho
=0.8$ and $n=4$. We will again assume that $p_{0}=p_{1}=1/2$, the unbiased
case. When considering upper bounds, we wished to use some $I\subsetneq
\lbrack 0,1]$. In this case we will use $I=[0,1]$. There are 16 images of 
$S_{\sigma }([0,1])$ for $|\sigma |=4$. These are listed in Table 
\ref{tab:lower} in increasing order of the left endpoint.

\begin{table}[tbp]
\begin{tabular}{l|l}
$\sigma$ & $S_\sigma([0, 1]$ \\ \hline
{}0000 & [.0000, .4096] \\ 
{}0001 & [.1024, .5120] \\ 
{}0010 & [.1280, .5376] \\ 
{}0100 & [.1600, .5696] \\ 
{}1000 & [.2000, .6096] \\ 
{}0011 & [.2304, .6400] \\ 
{}0101 & [.2624, .6720] \\ 
{}0110 & [.2880, .6976] \\ 
{}1001 & [.3024, .7120] \\ 
{}1010 & [.3280, .7376] \\ 
{}1100 & [.3600, .7696] \\ 
{}0111 & [.3904, .8000] \\ 
{}1011 & [.4304, .8400] \\ 
{}1101 & [.4624, .8720] \\ 
{}1110 & [.4880, .8976] \\ 
{}1111 & [.5904, 1.000]
\end{tabular}
\caption{Images of $S_\protect\sigma([0, 1])$ for $|\protect\sigma|=4$}
\label{tab:lower}
\end{table}

It is easy to compute that $\sup_{y\in \lbrack 0,1]}\mathcal{N}_{4}(y)=
\mathcal{N}_{4}(0.5)=14/2^{4}$. From this we conclude that $\llocdim
\mu (x)\geq \frac{\log 14/2^{4}}{4\log (0.8)} \sim 0.5984102692$. 
Because we are using $[0,1]$ exactly, we need to worry about situations
where $S_{\sigma }(0)=S_{\tau }(1)$ for some $|\sigma |=|\tau |$. Such cases
are known as \emph{transition points}. In these cases the value of $\sup_{y}
\mathcal{N}_{n}(y)$ may change as $\rho $ is increased or decreased
slightly. Such transition points occur when $\rho $ satisfies very precise
algebraic conditions and are easily enumerated for each $n$. A discussion of
how to find and properly compute these transition points for a fixed $n$ is
discussed in \cite{HareSidorov1, HareSidorov2}.

In Figure \ref{fig:lower} we indicate the lower bounds using $n$ up to 10.
The points on this graph indicate the lower bounds at transition points,
whereas the lines indicate regions between transition points when 
$\sup_{y\in \lbrack 0,1]}\mathcal{N}_{n}(y)$ is constant. We have computed
the lower bounds for $\sup_{y\in \lbrack 0,1]}\mathcal{N}_{n}(y)$ at all
transition points. 
The following theorem illustrates the kind of information this approach will
yield.

\begin{thm}
\label{thm:lower} Suppose the unbiased Bernoulli convolution $\mu _{\rho }$
satisfies the asymptotically weak separation condition. Then for all $x\in
\lbrack 0,1]$ we have lower bounds as described in Table \ref{tab:lower}.
More precise information can be found in Figure \ref{fig:lower}.

\begin{table}[tbp]
\begin{tabular}{l|l}
Range of $\rho$ & Lower bound for $\llocdim \mu(x)$
\\ \hline
$[0.50, 0.55]$ & $0.792021$ \\ 
$[0.55, 0.60]$ & $0.825663$ \\ 
$[0.60, 0.65]$ & $0.840348$ \\ 
$[0.65, 0.70]$ & $0.824701$ \\ 
$[0.70, 0.75]$ & $0.750984$ \\ 
$[0.75, 0.80]$ & $0.635012$ \\ 
$[0.80, 0.851]$ & $0.416226$
\end{tabular}
\caption{Lower bound for local dimensions}
\label{tab:lower}
\end{table}
\end{thm}

It is worth noting that Theorem \ref{thm:lower} applies only to those $\rho $
for which the measure satisfies the asymptotically wsc, whereas this is not
clearly indicated in Figure \ref{fig:lower}. Figure \ref{fig:lower} should
be interpreted as: If $\mu _{\rho }$ satisfies the asymtotically wsc, then
the value on the graph is a lower bound
of the set $\{\locdim \mu(x) : x \in [0,1]\}$. 
If, instead, $\mu _{\rho }$ does
not satisfy the asymptotically wsc, then the corresponding value on the
graph has no meaning. 

The smallest known Salem number is approximately $1.176280$, the root of 
$x^{10}+x^{9}-x^{7}-x^{6}-x^{5}-x^{4}-x^{3}+x+1$. This has an approximate
reciprocal of $0.850137$. The smallest Pisot number is approximately 
$1.324718,$ the root of $x^{3}-x-1$, with approximate reciprocal of $0.754877$.
As the only known Bernoulli convolutions satisfying the asymptotically wsc
are those where $\rho $ is the reciprocal of a Pisot or Salem number, Table 
\ref{tab:lower} was restricted to $\rho \leq 0.851$.

\begin{figure}[tbp]
\includegraphics[width=300pt,height=400pt,angle=270]{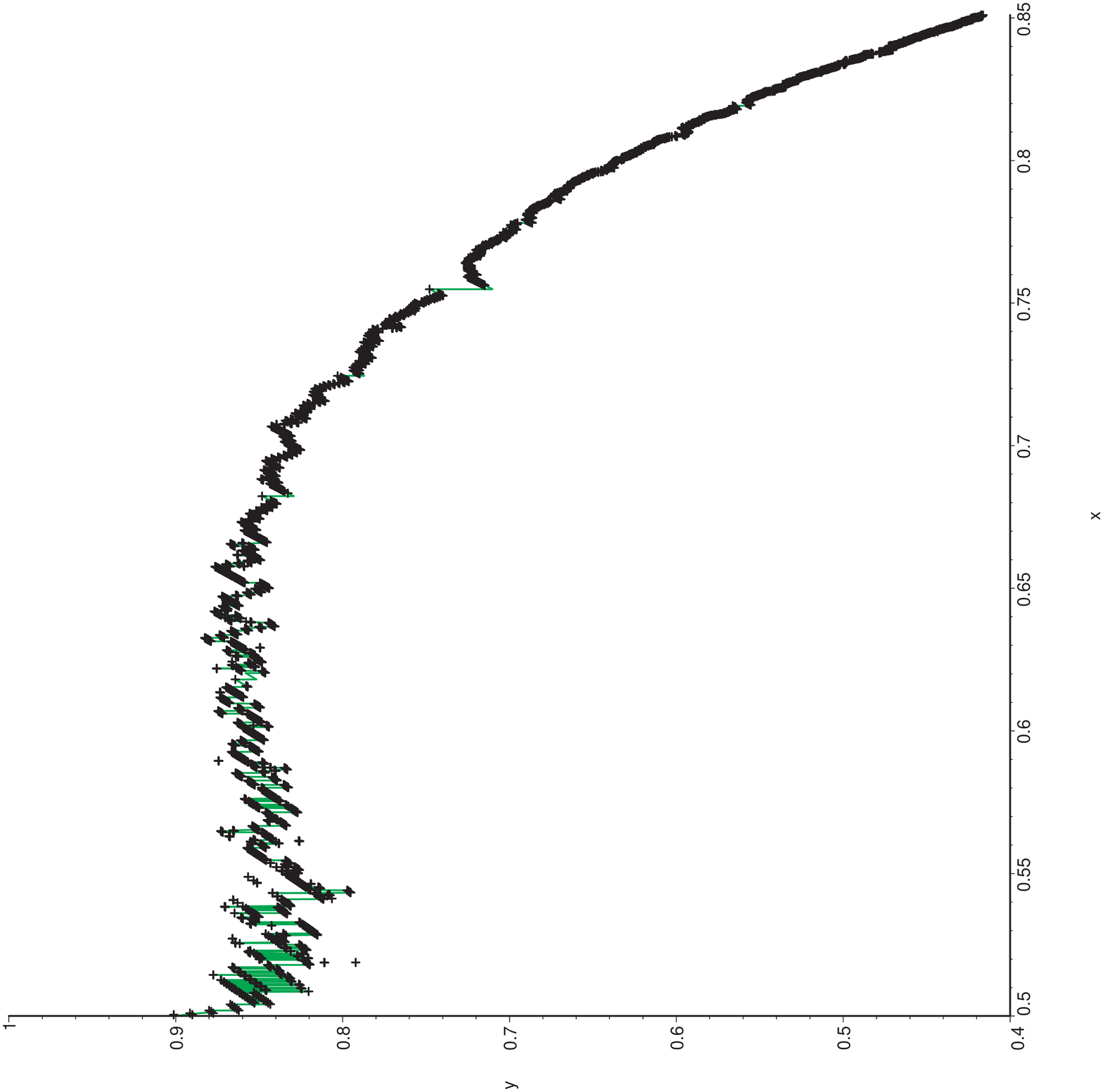}
\caption{Lower Bounds for local dimensions of unbiased Bernoulli
convolutions }
\label{fig:lower}
\end{figure}

\end{document}